\documentclass[12pt]{amsart}
\usepackage{latexsym,amssymb,amscd}

\def\NZQ{\Bbb}               
\def\NN{{\NZQ N}}

\def\B'c{{\mathcal{B'}}}
\def\U'c{{\mathcal{U'}}}

\def\opn#1#2{\def#1{\operatorname{#2}}} 
\opn\chara{char}
\opn\length{\ell}
\opn\projdim{proj\,dim}
\opn\injdim{inj\,dim}
\opn\ini{in}
\opn\rank{rank}
\opn\depth{depth}
\opn\sdepth{sdepth}
\opn\height{ht}
\opn\embdim{emb\,dim}
\opn\codim{codim}

\opn\Tr{Tr}
\opn\bigrank{big\,rank}
\opn\superheight{superheight}\opn\lcm{lcm}
\opn\trdeg{tr\,deg}%
\opn\reg{reg}
\opn\lreg{lreg}
\opn\set{set}
\opn\supp{Supp}
\opn\shad{Shad}
\opn\div{div}
\opn\Div{Div}
\opn\cl{cl}
\opn\Cl{Cl}
\opn\Spec{Spec}
\opn\Supp{Supp}
\opn\supp{supp}
\opn\Sing{Sing}
\opn\Ass{Ass}
\opn\Min{Min}
\opn\size{size}
\opn\bigsize{bigsize}
\opn\lex{lex}
\opn\Ann{Ann}
\opn\Rad{Rad}
\opn\Soc{Soc}
\opn\Ker{Ker}
\opn\Coker{Coker}
\opn\Im{Im}
\opn\Hom{Hom}
\opn\Tor{Tor}
\opn\Ext{Ext}
\opn\End{End}
\opn\Aut{Aut}
\opn\id{id}

\opn\nat{nat}
\opn\GL{GL}
\opn\SL{SL}
\opn\mod{mod}
\opn\ord{ord}
\opn\aff{aff}
\opn\con{conv}
\opn\relint{relint}
\opn\st{st}
\opn\lk{lk}
\opn\cn{cn}
\opn\core{core}
\opn\vol{vol}
\opn\gr{gr}

\def\pot#1#2{#1[\kern-0.28ex[#2]\kern-0.28ex]}

\opn\dirlim{\underrightarrow{\lim}}
\opn\invlim{\underleftarrow{\lim}}
\let\Dirsum=\bigoplus

\def\pnt{{\raise0.5mm\hbox{\large\bf.}}}

\def\Implies{\ifmmode\Longrightarrow \else
     \unskip${}\Longrightarrow{}$\ignorespaces\fi}
\def\implies{\ifmmode\Rightarrow \else
     \unskip${}\Rightarrow{}$\ignorespaces\fi}
\def\iff{\ifmmode\Longleftrightarrow \else
     \unskip${}\Longleftrightarrow{}$\ignorespaces\fi}

\let\:=\colon
\newtheorem{Theorem}{Theorem}[section]
\newtheorem{Lemma}[Theorem]{Lemma}
\newtheorem{Corollary}[Theorem]{Corollary}
\newtheorem{Proposition}[Theorem]{Proposition}
\newtheorem{Remark}[Theorem]{Remark}

\newtheorem{Example}[Theorem]{Example}

\let\epsilon=\varepsilon
\let\phi=\varphi
\let\kappa=\varkappa
\numberwithin{equation}{section}

\title{Monomial ideals of minimal depth and trivial modifications}

\author[Muhammad Ishaq]{Muhammad Ishaq}
\email{ishaq$\_\,$maths@yahoo.com}
\begin{document}
\maketitle
\begin{abstract} Let $S$ be a polynomial algebra over a field. We study classes of monomial ideals (as for example lexsegment ideals) of $S$ having minimal depth. In particular, Stanley's conjecture holds for these ideals. Also we show that if Stanley's conjecture holds for a square free monomial ideal then it holds for all its trivial modifications.

\end{abstract}
\section*{Introduction}
Let $K$ be a field and $S=K[x_1,\ldots,x_n]$ be a polynomial
ring in $n$ variables over $K$. Let $I\subset S$ be a monomial ideal and $I=\cap_{i=1}^sQ_i$ an irredundant primary decomposition
of $I$, where the $Q_i$ are monomial ideals. Let $Q_i$ be $P_i$-primary. Then each $P_i$ is a monomial prime ideal and $\Ass(S/I)=\{P_1,\dots,P_s\}$.\\
\indent According to Lyubeznik \cite{LB} the size of $I$, denoted $\size(I)$, is the
number $a+(n-b)-1$, where $a$ is the minimum number $t$ such that there exist $j_1<\dots <j_t$ with $$\sqrt{\sum_{l=1}^tQ_{j_l}}=\sqrt{\sum_{j=1}^sQ_j},$$ and where $b=\height(\sum_{j=1}^sQ_j).$ It is clear from the definition that $\size(I)$ depends only on the associated prime ideals of $S/I$. In the above definition if we replaced ``there exists $j_1<\dots <j_t$" by ``for all $j_1<\dots <j_t$", we obtain the definition of $\bigsize(I)$, introduced by Popescu \cite{D3}. Clearly $\bigsize(I)\geq \size(I)$.
\begin{Theorem}(Lyubeznik \cite{LB})\label{dsize} Let $I\subset S$ ba a monomial ideal then $\depth(I)\geq 1+\size(I).$
\end{Theorem}
\noindent Herzog, Popescu and Vladoiu  say in \cite{HPV} that a monomial ideal $I$  has \emph{minimal depth}, if $\depth(I)=\size(I)+1$. Suppose above that $P_i\not\subset \sum_{1=j\neq i}^sP_j$ for all $i\in [s]$. Then $I$ has minimal depth as shows our Corollary \ref{p} which extends \cite[Theorem 2.3]{D3}. It is easy to see that if $I$ has bigsize $1$ then it must have minimal depth (see our Corollary \ref{bigsize}).

Next we  consider the lexicographical order on the monomials of $S$ induced by $x_1>x_2>\dots>x_n$. Let $d\geq2$ be an integer and $\mathcal{M}_d$ the set of monomials of degree $d$ of $S$. For two monomials $u,v\in \mathcal{M}_d$, with $u\geq_{lex}v$, the set
$$\mathcal{L}(u,v) =\{w\in \mathcal{M}_d| u\geq_{lex}w\geq_{lex}v\}$$
is called a lexsegment set. A lexsegment ideal in $S$ is a monomial ideal of
$S$ which is generated by a lexsegment set. We show that a lexsegment ideal has minimal depth (see our Theorem \ref{thelex}).

Now, let $M$ be a finitely generated multigraded $S$-module. Let $M$ be an $S$-module, $z\in M$ be a homogeneous element in $M$ and $zK[Z]$, $Z\subseteq \{x_1,\ldots,x_n\}$ the linear $K$-subspace of $M$ of all elements $zf$, $f\in K[Z]$. Such a linear $K$-subspace $zK[Z]$ is called a Stanley space of dimension $|Z|$ if it is a free $K[Z]$-module, where $|Z|$ denotes the number of indeterminates in $Z$. A presentation of $M$ as a finite direct sum of spaces $\mathcal{D}:\,\,M=\Dirsum_{i=1}^r z_iK[Z_i]$ is called a Stanley decomposition. Stanley depth of a decomposition $\mathcal{D}$ is the number
$$\sdepth \mathcal{D}=\min\{|Z_i|:i=1,\ldots,r\}.$$ The number
\[
\sdepth(M):=\max\{\sdepth({\mathcal D}):\text{Stanley decomposition of}\;M\}
\]
is called Stanley depth of $M$. In \cite{RP} R. P. Stanley conjectured that $$\sdepth(M)\geq \depth(M).$$

\begin{Theorem}[\cite{HPV}]\label{size} Let $I\subset S$ be a monomial ideal then $\sdepth(I)\geq 1+\size(I).$ In particular, Stanley's conjecture holds for the monomial ideals of minimal depth.
\end{Theorem}
 As a consequence, Stanley's depth holds for all ideals considered above since they have minimal depth. It is still not known a relation between $\sdepth(I)$ and $\sdepth(S/I)$, but our Theorem \ref{theorem1} shows that Stanley's conjecture holds also for $S/I$ if $P_i\not\subset \sum_{1=j\neq i}^sP_j$ for all $i\in [s]$.

Let $J$ be a so called trivial modification of a square free monomial ideal $I$ in the sense of \cite{HTT}, \cite{SN}.
Our Theorem \ref{theorem2} shows that $\sdepth(J)=\sdepth(I)$. It follows that if Stanley's conjecture holds for $I$ then it holds for all trivial modifications of it (see our Corollary \ref{c}).
\section{Minimal depth}
 We start this section  extending some results of Popescu in \cite{D3}. Lemma \ref{Lemma1}, Proposition \ref{pro1}, Lemma \ref{size1} and Corollary \ref{bigsize} were proved by Popescu when $I$ is a squarefree monomial ideal. We show that with some small changes the same proofs work even in the non-squarefree case.
\begin{Lemma}\label{Lemma1}
Let $I=\bigcap\limits_{i=1}^sQ_i$ be the irredundant presentation of $I$ as an intersection of primary monomial ideals. Let $P_i:=\sqrt{Q_i}$. If $P_s\not\subset \sum_{i=1}^{s-1}P_i$, then $$\depth(S/I)=\min\{\depth(S/\cap_{i=1}^{s-1}Q_i),\depth(S/Q_s),\,1+\depth(S/\cap_{i=1}^{s-1}(Q_i+Q_s))\}.$$
\end{Lemma}
\begin{proof}
We have the following exact sequence $$0\longrightarrow S/I\longrightarrow S/\cap_{i=1}^{s-1}Q_i\oplus S/Q_s\longrightarrow S/\cap_{i=1}^{s-1}(Q_i+Q_s)\longrightarrow 0.$$ Clearly $\depth(S/I)\leq \depth(S/Q_s)$ by \cite[Proposition 1.2.13]{BH}. Choosing $x_j^a$ where $x_j\in P_s\not\subset \sum_{i=1}^{s-1}P_i$ and $a$ is minimum such that $x_j^a\in Q_s$ we see that $I:x_j^a=\cap_{i=1}^{s-1}Q_i$ and by \cite[Corollary 1.3]{R1} we have $$\depth(S/I)\leq \depth S/(I:x_j^a)=\depth S/(\cap_{i=1}^{s-1}Q_i).$$ Now by using Depth Lemma (see \cite[Lemma 1.3.9]{RV}) we have $$\depth(S/I)=\min\{\depth (S/\cap_{i=1}^{s-1}Q_i),\depth(S/Q_s),\,1+\depth (S/\cap_{i=1}^{s-1}(Q_i+Q_s))\},$$ which is enough.
\end{proof}
\begin{Proposition}\label{pro1}
Let $I=\bigcap\limits_{i=1}^sQ_i$ be the irredundant presentation of $I$ as an intersection of primary monomial ideals. Let $P_i:=\sqrt{Q_i}$. If $P_i\not\subset \sum_{1=i\neq j}^{s-1}P_j$ for all $i\in [s]$. Then $\depth(S/I)=s-1$.
\end{Proposition}
\begin{proof}
It is enough to consider the case when $\sum_{j=1}^sP_j=\mathfrak{m}$. We use induction on $s$. If $s=1$ the result is trivial. Suppose that $s>1$. By Lemma \ref{Lemma1} we get $$\depth(S/I)=\min\{\depth(S/\cap_{i=1}^{s-1}Q_i),\depth(S/Q_s),\,1+\depth(S/\cap_{i=1}^{s-1}(Q_i+Q_s))\}.$$ Then by induction hypothesis we have \[\depth(S/\cap_{i=1}^{s-1}Q_i)=s-2+\dim (S/(\sum_{i=1}^{s-1}Q_i))\geq s-1.\] We see that $\cap_{i=1}^{s-1}(Q_i+Q_s)$ satisfies also our assumption, the induction hypothesis gives $\depth (S/\cap_{i=1}^{s-1}(Q_i+Q_s))=s-2$. Since $Q_i\not\subset Q_s$, $i<s$ by our assumption we get $\depth(S/Q_s)>\depth(S/(Q_i+Q_s))$ for all $i<s$. It follows $\depth(S/Q_s)\geq 1+\depth(S/\cap_{i=1}^{s-1}(Q_i+Q_s))$ which is enough.
\end{proof}
\begin{Corollary}\label{p}
Let $I\subset S$ be a monomial ideal such that $\Ass(S/I)=\{P_1\dots,P_s\}$ where $P_i\not\subset \sum_{1=j\neq i}^sP_j$ for all $i\in [s]$. Then $I$ has minimal depth.
\end{Corollary}
\begin{proof}
Clearly $\size(I)=s-1$ and by Proposition \ref{pro1} we have $\depth(I)=s$, thus we have $\depth(I)=\size(I)+1$, i.e. $I$ has minimal depth.
\end{proof}

\begin{Lemma}\label{size1}
Let $I=\cap_{i=1}^sQ_i$ be the irredundant primary decomposition of $I$ and $\sqrt{Q_i}\neq \mathfrak{m}$ for all $i$. Suppose that there exists $1\leq r<s$ such that $\sqrt{Q_i+Q_j}=\mathfrak{m}$ for each $r<j\leq s$ and $1\leq i\leq r$. Then $\depth(I)=2$.
\end{Lemma}
\begin{proof}
The proof follows by using Depth Lemma on the following exact sequence.$$0\longrightarrow S/I\longrightarrow S/\cap_{i=1}^rQ_i\oplus S/\cap_{j>r}^sQ_j\longrightarrow S/\cap_{i=1}^r\cap_{j>r}^s(Q_i+Q_j)\longrightarrow 0.$$
\end{proof}
\begin{Corollary}\label{bigsize}
Let $I\subset S$ be a monomial ideal. If bigsize of $I$ is one then $I$ has minimal depth.
\end{Corollary}
\begin{proof}
We know that $\size(I)\leq \bigsize(I)$. If $\size(I)=0$ the $\depth(I)=1$ and the result follows in this case. Now let us suppose that $\size(I)=1$. By Lemma \ref{size1} we have $\depth(I)=2$. Hence the result follows.
\end{proof}

\indent Let $d\geq2$ be an integer and $\mathcal{M}_d$ the set of monomials of degree $d$ of $S$. For two monomials $u,v\in \mathcal{M}_d$, with $u\geq_{lex}v$, we consider the lexsegment set
$$\mathcal{L}(u,v) =\{w\in \mathcal{M}_d| u\geq_{lex}w\geq_{lex}v\}.$$

\begin{Theorem}\label{thelex}
Let $I=(\mathcal{L}(u,v))\subset S$ be a lexsegment ideal.  Then $\depth(I)=\size(I)+1$, that is $I$ has minimal depth.
\end{Theorem}
\begin{proof}
For the trivial cases $u=v$ the result is obvious. Suppose that $u=x_1^{a_1}\cdots x_n^{a_n},\\
 v=x_1^{b_1}\cdots x_n^{b_n}\in S$.
 First assume that $b_1=0$. If there exist $r$ such that $a_1=\dots =a_r=0$ and $a_{r+1}\neq 0$, then $I$ is a lexsegment ideal in $S':=K[x_{r+1},\dots,x_{n}]$. We get  $\depth(IS)=\depth(IS')+r$ and by definition of size we have $\size(IS)=\size(IS')+r$. This means that without loss of generality we can assume that $a_1>0$. If $x_nu/x_1\geq_{lex}v$, then by \cite[Proposition 3.2]{EOS} $\depth(I)=1$ which implies that $\mathfrak{m}\in \Ass(S/I)$, thus $\size(I)=0$ and the result follows in this case. Now consider the complementary case $x_nu/x_1<_{lex}v$, then $u$ is of the form $u=x_1x_l^{a_l}\cdots x_n^{a_n}$ where $l\geq 2$. Let $I=\cap_{i=1}^sQ_i$ be an irredundant primary decomposition of $I$, where $Q_i's$ are monomial primary ideals. If $l\geq 4$ and $v= x_2^{d}$ then by \cite[Proposition 3.4]{EOS} we have $\depth(I)=l-1$. After \cite[Proposition 2.5($ii$)]{MI2} we know that $$\sqrt{\sum_{i=1}^sQ_i}=(x_1,x_2,x_l,\dots,x_n)\notin \Ass(S/I),$$ but $(x_1,x_2),(x_2,x_l,\dots,x_n)\in\Ass(S/I)$. Therefore, $\size(I)=l-2$ and we have $\depth(I)=\size(I)+1$, so we are done in this case. Now consider the case $v=x_2^{d-1}x_j$ for some $3\leq j\leq n-2$ and $l\geq j+2$, then again by \cite[Proposition 3.4]{EOS} we have $\depth(I)=l-j+1$ and by \cite[Proposition 2.5($ii$)]{MI2} we have $$\sqrt{\sum_{i=1}^sQ_i}=(x_1,\dots,x_j,x_l,\dots,x_n)\notin \Ass(S/I)$$ and $(x_1,\dots,x_j),(x_2,\dots,x_j,x_l,\dots,x_n)\in \Ass(S/I)$. Therefore, $\size(I)=l-j$ and again we have $\depth(I)=\size(I)+1$. Now for all the remaining cases by \cite[Proposition 3.4]{EOS} we have $\depth(I)=2$, and by \cite[Proposition 2.5($i$)]{MI2} $$\sqrt{\sum_{i=1}^sQ_i}=(x_1,\dots,x_n)\notin \Ass(S/I),$$ but $(x_1,\dots,x_j),(x_2,\dots,x_n)\in \Ass(S/I),$ for some $j\geq 2.$ Therefore $\size(I)=1$. Thus the equality $\depth(I)=\size(I)+1$ follows in all cases when $b_1=0$.   \\

 \indent Now let us consider that $b_1>0$, then $I=x_1^{b_1}I'$ where $I'=(I:x_1^{b_1})$. Clearly $I'$ is a lexsegment ideal generated by the lexsegment set $\mathcal{L}(u',v')$ where $u'=u/x_1^{b_1}$ and $v'=v/x_1^{b_1}$. The ideals $I$, $I'$ are isomorphic, therefore $\depth(I')=\depth(I)$. It is enough to show that $\size(I')=\size(I)$. We have the exact sequence $$0 \rightarrow S/I'\mathop\rightarrow \limits_{}^{x_1^{b_1}} S/I\rightarrow S/(I,x_1^{b_1})=S/(x_1^{b_1})\rightarrow 0,$$ and therefore $$\Ass(S/I')\subset \Ass(S/I)\subset \Ass(S/I')\cup \{(x_1)\}.$$ As $\{(x_1)\}\in \Ass(S/I)$ since it is a minimal prime over $I$, we get $\Ass (S/I)=\Ass (S/I')\cup \{(x_1)\}$. Let $s'$ be the minimum number such that there exist $P_1,\ldots,P_{s}\in \Ass S/I'$ such that $\sum_{i=1}^s P_i=a:=\sum_{P\in \Ass (S/I')} P$. Then $\size(I')=s'+ \dim (S/a)-1$. Let $s$ be the minimum number $t$ such that there exist $t$ prime ideals in $\Ass (S/I)$ whose sum is $(a,x_1)$. By \cite[Lemma 2.1]{MI2} we have that atleast one prime ideal from $\Ass(S/I')$ contains necessarily $x_1$, we have $x_1\in a$. It follows $s\leq s'$ because anyway $\sum_{i=1}^{s'}P_i=a=\sum_{P\in \Ass(S/I)}P$. If we have $P_1',\dots,P_{s-1}'\in \Ass(S/I')$ such that $\sum_{i=1}^{s-1}P_{i}'+(x_1)=a$ then we have also $\sum_{i=1}^{s-1}P_i'+P_1=a$ for some $P_1\in \Ass(S/I')$ which contains $x_1$. Thus $s=s'$ and so $\size(I)=\size(I')$.
 \end{proof}
\section{Stanley depth and trivial modifications}
Using Corollaries \ref{p}, \ref{bigsize} and Theorems \ref{thelex}, \ref{size} we get the following theorem.
\begin{Theorem} Stanley's conjecture holds for $I$, if it satisfies one of the following statements:
\begin{enumerate}
\item{} $P_i\nsubseteq \sum_{1=j\neq i}^sP_j$ for all $i\in [s]$,
\item{}the bigsize of $I$ is one,
\item{} $I$ is a lexsegment ideal.
\end{enumerate}
\end{Theorem}
\begin{Remark}{\em
Usually, if Stanley's conjecture holds for an ideal $I$ then we may show that it holds for the module $S/I$ too. There exist no general explanation for this fact. If $I$ is a monomial ideal of bigsize one then Stanley's conjecture holds for $S/I$. Indeed, case $\depth(S/I)=0$ is trivial. Suppose $\depth(S/I)\neq 0$, then by Lemma \ref{size1} $\depth(S/I)=1$, therefore by \cite[Theorem 2.1]{MC} $\sdepth(S/I)\geq 1$. If $I$ is a lexsegment ideal then Stanley's conjecture holds for $S/I$ \cite{MI2}. Below we show this fact in the first case of the above theorem.}
\end{Remark}
\begin{Theorem}\label{theorem1}
Let $I=\bigcap\limits_{i=1}^sQ_i$ be the irredundant presentation of $I$ as an intersection of primary monomial ideals. Let $P_i:=\sqrt{Q_i}$. If $P_i\not\subset \sum_{1=i\neq j}^{s-1}P_j$ for all $i\in [s]$
 then $\sdepth(S/I)\geq \depth(S/I)$,
that is the Stanley's conjecture holds for $S/I$.
\end{Theorem}
\begin{proof}
 Using \cite[Lemma 3.6]{HVZ} it is enough to consider the case $\sum_{i=1}^sP_i=\mathfrak{m}$. By Proposition \ref{pro1} we have $\depth(S/I)=s-1$. We show that $\sdepth(S/I)$ $\geq s-1$.  Apply induction on $s$,  case $s=1$ being clear. Fix $s>1$ and apply induction on $n$. If $n\leq 5$ then the result follows by \cite{P}. Let ${A}:=\cup_{i=1}^s(G(P_i)\setminus \sum_{1=j\neq i}^sG(P_j))$. If $(A)=\mathfrak{m}$ then  note that $G(P_i)\cap G(P_j)=\emptyset$ for all $i\not= j$.  By \cite[Theorem 2.1]{MI1} we have $\sdepth(S/I)\geq s-1$. Now suppose that $(A)\not =\mathfrak{m}$. By renumbering the primes and variables we can assume that $x_n \not \in A$. There exists a number $r$, $2\leq r\leq s$ such that $x_n\in G(P_j)$, $1\leq j\leq r$ and $x_n\notin G(P_j)$, $r+1\leq j\leq s$.
 Let $S':=K[x_1,\ldots,x_{n-1}]$. First assume that $r<s$. Let $Q'_j=Q_j\cap S'$, $P'_j=P_j\cap S'$ and $J=\bigcap_{i=r+1}^s Q'_i\subset S'$, $L=\bigcap_{i=1}^r Q'_i\subset S'$. We have  $(I,x_n)=((J\cap L),x_n)$
because  $(Q_j,x_n)=(Q'_j,x_n)$ using the structure of monomial primary ideals given in \cite{RV}. In the exact sequence $$0\longrightarrow{S}/{(I:x_n)}\longrightarrow {S}/{I}\longrightarrow {S}/{(I,x_n)}\longrightarrow 0,$$
the sdepth of the right end is $\geq s-1$ by induction hypothesis on $n$ for  $J\cap L \subset S'$ (note that we have    $P'_i\not\subset \sum_{1=i\neq j}^{s-1}P'_j$ for all $i\in [s]$ since $x_n\not\in A$).
 Let $e_I$ be the maximum degree in $x_n$ of a monomial from $G(I)$. Apply induction on $e_I$. If $e_I=1$ then $(I:x_n)=JS$ and
the sdepth of the left end in the above exact sequence  is equal with $\sdepth(S/JS)\geq (s-r-1)+r=s-1$ since there are at least $r$ variables which do not divide the minimal monomial generators of ideal $(I:x_n)$ and we may apply induction hypothesis on $s$ for $J$. By \cite[Theorem 3.1]{R1} we have
$\sdepth(S/I)\geq \min\{\sdepth(S/(I:x_n)), \sdepth (S/(I,x_n))\}\geq s-1$. If $e_I>1$ then note that $e_ {(I:x_n)}<e_I$ and by induction hypothesis on $e_I$ or $s$ we get $\sdepth(S/(I:x_n))\geq s-1$. As above we obtain by \cite[Theorem 3.1]{R1} $\sdepth(S/I)\geq s-1$.

 Now let $r=s$. If $e_I=1$ then $I=(L,x_n)$ and  by induction on $n$ we have $\sdepth(S/I)=\sdepth(S'/L)\geq s-1$. If $e_I>1$ then
 by induction hypothesis on $e_I$ and $s$ we get $\sdepth(S/(I:x_n))\geq s-1$. As above we are done using \cite[Theorem 3.1]{R1}.
\end{proof}

Let $I$ be a squarefree monomial ideal with minimal monomial generating set $G(I)=\{u_1,\dots,u_m\}$. Let $u$ be a monomial of $S$ then $\supp(u):=\{i:x_i \text{ divides } u\}$. Then we call a monomial ideal $J$ a {\em modification} of $I$ (see \cite{SN}), if $G(J)=\{v_1,\dots,v_m\}$ and $\supp(v_i)=\supp(u_i)$ for all $i$. Obviously, $\sqrt{J}=I$. Let $\alpha=(a_1,\ldots,a_n) \in \NN^n$, $a_i\neq 0$ for all $i$ and $\sigma_{\alpha}$ be the flat $K$-morphism of $S$ given by $x_i\rightarrow x_i^{a_i}$, $i\in [n]$. Let $I^{\alpha}:=\sigma_{\alpha}(I)S$. Then $I^{\alpha}$ is called a {\em trivial modification} of $I$.
It is well known that $\sdepth(I^{\alpha})\leq \sdepth(I)$ by \cite[Corollary 2.2]{MI}. In our next theorem we will show that the equality holds in this case.
\begin{Example}{\em
Let $I=(x_1x_2x_3,x_2x_4,x_4x_5x_6,x_2x_6,x_5x_7,x_1x_2x_6x_7)\subset K[x_1,\dots,x_7]$ and $\alpha=(2,3,6,3,7,8,2)$, then we have $$I^\alpha=(x_1^2x_2^3x_3^6,x_2^3x_4^3,x_4^3x_5^7x_6^8,x_2^3x_6^8,x_5^7x_7^2,x_1^2x_2^3x_6^8x_7^2).$$
}
\end{Example}
\begin{Lemma}[\cite{C},\cite{IQ1}]\label{LIQ}
Let $r,m$ and $a$ be positive integers with $r<m$ and $v_1,\dots,v_m \in K[x_2,\dots,x_n]$ be some monomials of $S$. Let $I=(x_1^{a}v_1,\dots,x_1^{a}v_r,v_{r+1},\dots,v_{m})$ and $I'=(x_1^{a+1}v_1,\dots, x_1^{a+1}v_r,v_{r+1},\dots,v_m)$ be monomial ideals of $S$. Then $$\sdepth(I)=\sdepth(I').$$
\end{Lemma}
\begin{Theorem}\label{theorem2}
Let $\alpha \in \NN^n$, then $\sdepth(I^\alpha)=\sdepth(I)$.
\end{Theorem}
\begin{proof}
Apply induction on $s=\sum_{i=1}^na_i$. If $s=n$ there exist nothing to show. Suppose that $s>n$. Then there exists $i$ such that $a_i>1$, let us say $a_1>1$. Renumbering the variables we may suppose $I^\alpha$ has the form $(x_1^{a_1}v_1,\ldots,x_r^{a_1}v_r,v_{r+1},\ldots,v_m)$ for some monomials $v_1,\ldots,v_m$ which are not multiple of $x_1$. By Lemma \ref{LIQ} we get $\sdepth(I^\alpha)=\sdepth(J)$ for $J=(x_1^{a_1-1}v_1,\ldots,x_1^{a_1-1}v_r,v_{r+1},\ldots,x_m)$. But $J=I^{\alpha'}$ for $\alpha'=(a_1-1,a_2,\ldots,a_n)$ which has $s'=s-1$. By induction hypothesis we have $\sdepth(I^{\alpha'})=\sdepth(I)$, which is enough.
\end{proof}
\begin{Corollary}\label{c}
Let $I\subset S$ be a squarefree monomial ideal if the Stanley conjecture holds for $I$, then the Stanley conjecture also holds for $I^\alpha$.
\end{Corollary}
\begin{proof}
Since $\depth(I)\leq \sdepth(I)$, by \cite[Theorem 2.3]{HTT} and Theorem \ref{theorem2} we have $\depth(I^\alpha)\leq \depth(I)\leq \sdepth(I)=\sdepth(I^\alpha)$. This completes the proof.
\end{proof}

\end{document}